\documentclass[a4paper,11pt,reqno]{amsart}

\usepackage{amsmath}%
\usepackage{amsfonts}%
\usepackage{amssymb}%
\usepackage{graphicx}
\usepackage{float}
\usepackage{amssymb}
\usepackage{color}
\usepackage{dsfont}
\usepackage{setspace}
\usepackage{cite}

\theoremstyle{plain}

\newtheorem{theorem}{Theorem}
\newtheorem{definition}[theorem]{Definition}

\newtheorem{proposition}[theorem]{Proposition}
\newtheorem{remark}[theorem]{Remark}

\allowdisplaybreaks

\newcommand\NN{\mathbb{N}}

\newcommand\RR{\mathbb{R}}
\newcommand\ex{\mathbb{E}}

\newcommand\eins{\mathrm{1}}

\numberwithin{equation}{section}

\begin{document}

\title[Extensions of the Hitsuda-Skorokhod integral]{Extensions of the Hitsuda-Skorokhod integral}

\author[P. Parczewski]{Peter Parczewski}

%\thanks{\textit{Address:} \textsc{University of Mannheim, Institute of Mathematics A5, 6, D-68131 Mannheim, Germany}\\ \textit{E-mail address:}  parczewski@math.uni-mannheim.de}

\address{University of Mannheim, Institute of Mathematics A5, 6, D-68131 Mannheim, Germany.}
\email{parczewski@math.uni-mannheim.de}

\date{\today}

\begin{abstract}
We present alternative definitions of the stochastic integral introduced by Ayew and Kuo and of the Hitsuda-Skorokhod integral extended to domains in $L^p$-spaces, $p \geq 1$. Our approach is motivated by the S-transform characterization of the Hitsuda-Skorokhod integral and based on simple processes of stochastic exponential type. We prove that the new stochastic integral extends the mentioned stochastic integrals above and we outline their connection.
\end{abstract}



\maketitle

\section{Introduction}

A series of articles starting from the work by Ayew and Kuo \cite{AyedKuo1, AyedKuo2} establishes a new stochastic integral with respect to a Brownian motion extending the It\^o integral to nonadapted processes. 
Let $B= (B_{t})_{t \geq 0}$ be a Brownian motion on  a probability space $(\Omega, \mathcal{F}, P)$,  where the $\sigma$-field $\mathcal{F}$ is generated by the Brownian motion and completed by null sets. The augmented Brownian filtration is denoted by $(\mathcal{F}_t)_{t\geq 0}$.

The construction of the Ayew-Kuo stochastic integral on a compact interval $[a,b] \subset [0,\infty)$ is based on the consideration of the adapted part $(f(t))_{t \in [a,b]}$, and the remaining part $(\varphi(t))_{t \in [a,b]}$, the so-called instantly independent part, of the integrand, which means that the element $\varphi(t)$ is independent of $\mathcal{F}_t$ for every $t \in [a,b]$. 
%Both processes have continuous paths.
This decomposition of the integrand allows a simple construction of this Ayew-Kuo integral for continuous integrands $f \varphi$ via a Riemann-sum approach, see e.g. \cite{AyedKuo1,AyedKuo2,HKSZ}, similar to the Riemann-sum approach for the general It\^o integral (cf. \cite[Section 5.3]{KuoStochInt}). In particular, the Ayew-Kuo integral allows to integrate nonadapted processes $f\varphi \notin L^2(\Omega \times [a,b])$.

A classical extension of the It\^o integral to nonadapted integrands is the Hitsuda-Skorokhod integral, see e.g. \cite{Holden_Buch, Nualart}. A Riemann-sum approach to the Hitsuda-Skorokhod integral is given in \cite{NualartPardoux}. However, the Hitsuda-Skorokhod integral can be introduced by various simpler approaches, like Wiener chaos expansions \cite{Holden_Buch}, as the adjoint of the Malliavin derivative \cite{Nualart} or from the more general white noise integrals \cite{Kuo}. Whereas these approaches mostly depend on the Hilbert space structure of the space containing the integrands, there exist characterizations of the Hitsuda-Skorokhod integral via transforms, see e.g. \cite{Janson}.

One of these characterizations is as follows. We denote by $I(f)$ the Wiener integral of $f \in L^2([0,\infty))$  and by $\|\cdot\|$ the norm on this space. The process $u \in L^2(\Omega \times [0,\infty))$ is Skorokhod-integrable, if there exists an element $X \in L^2(\Omega)$ such that for all $f \in L^2([0,\infty))$,
\[
\ex[X e^{I(f) - \|f\|^2/2}] = \int_{0}^{\infty}  \ex[u_s e^{I(f) - \|f\|^2/2}] f(s) ds.
\]
Then $\delta(u) := X$ is the Hitsuda-Skorokhod integral. This characterization is mainly based on the fact that the linear span of the set $\{e^{I(f) - \|f\|^2/2}, f \in L^2([0,\infty))\}$ is dense in $L^2(\Omega)$.

Such transform characterizations of the Hitsuda-Skorokhod integral can be extended beyond integrands in $L^2(\Omega \times [0,\infty))$, see e.g. \cite[Chapter 16]{Janson}.

Due to the very different introductions, the relation between the Ayew-Kuo integral and the Hitsuda-Skorokhod integral is not clarified yet and stated as an open problem in the survey article on the Ayew-Kuo integral \cite{KuoNewInt}.

In this note we prove that the Ayew-Kuo integral equals the Hitsuda-Skorokhod integral on Skorokhod integrable processes in $L^2(\Omega \times [0,\infty))$. Moreover we propose alternative definitions of stochastic integrals which extend both, the Ayew-Kuo integral and the Hitsuda-Skorokhod integral to domains in $L^p(\Omega \times [0,\infty))$ for all $p \geq 1$. 

Our approach imitates the characterization of the Hitsuda-Skorokhod integral. Firstly, we define the stochastic integral on the set of simple processes
\[
u_s = e^{I(g) - \|g\|^2/2} h(s)
\]
for simple deterministic functions $g,h: [0,\infty) \rightarrow \RR$. Then, the stochastic integral is extended by linearity and density arguments to different domains. As the integrals of these nonadapted processes, defined via
\[
e^{I(g) - \|g\|^2/2}\left(I(h) - \int_{0}^{\infty} g(s)h(s) ds\right) ,
\]
are elements in $L^1(\Omega \times [0,\infty))$ and extend the It\^o integral elementary, the object of interest in the extension of the integral are the topologies involved. We observe that the Ayew-Kuo integral is extended via a version of an almost-sure-probability closedness which is more general than the $L^p(\Omega \times [0,\infty))$-$L^p(\Omega)$-closedness of the general Hitsuda-Skorokhod integral. 

The article is organized as follows. In Section \ref{section:StartingNewInt} we show that already a subclass of elementary integrands of the Ayew-Kuo integral is in the domain of the Hitsuda-Skorokhod integral and both stochastic integrals then coincide.

In Section \ref{section:Extensions} we propose alternative definitions of the extended Hitsuda-Skorokhod integral and of the Ayew-Kuo integral to domains in $L^p(\Omega \times [0,\infty))$ for $p \geq 1$. In our main result we prove that the new definition of the Ayew-Kuo integral extends the definition in \cite{KuoNewInt, HKSZ} as well as the extended Hitsuda-Skorokhod integral.

\section{Starting from the Ayew-Kuo stochastic integral}\label{section:StartingNewInt}

Firstly, we introduce the \emph{elementary Ayew-Kuo integral}. The following definition is taken from \cite{HKSZ}. 

\begin{definition}\label{def:ElemNewIntegral}
Let $[a,b] \subset [0,\infty)$. Suppose an adapted continuous stochastic process $\left(f(t)\right)_{t \in [a,b]}$ and an instantly independent continuous process $\left(\varphi(t)\right)_{t \in [a,b]}$, i.e. for all $t \in [a,b]$, $\varphi(t)$ is independent of $\mathcal{F}_t$. Then, provided the limit on the right hand side exists in probability, the Ayew-Kuo integral is given by
\begin{equation}\label{eq:NewIntegral1}
\bar{I}(f\varphi) := \int_{a}^{b} f(t) \varphi(t) dB_t = \lim_{\| \Delta_n\| \rightarrow 0} \sum_{i=1}^{n} f(t_{i-1}) \varphi(t_i) \Delta B_{i},
\end{equation}
where $\Delta_n = \{a = t_0 < t_1< \cdots < t_n = b\}$ with $\| \Delta_n\|= \max\limits_{i=1,\ldots, n} |t_i-t_{i-1}|$ and $\Delta B_i := B_{t_i} - B_{t_{i-1}}$.
\end{definition}

It is clear that the elementary Ayew-Kuo integral extends the It\^o integral to nonadapted integrands and exhibits the zero mean property. Moreover, the integrands are not necessary in $L^2(\Omega \times [a,b], \mathcal{F} \otimes \mathcal{B}([a,b]), P \otimes \lambda_{[a,b]})$.

A classical extension of the It\^o integral to nonadapted integrands with the additive noise structure is the Hitsuda-Skorokhod integral. There are several approaches for  the introduction of this stochastic integral. Essentially, these are via Wiener-It\^o chaos expansion, as the adjoint operator of the Malliavin derivative (cf. \cite[1.3]{Nualart}, \cite[7.3]{Janson}) or via the following transform characterizations motivated by white noise integrals, cp. \cite[Theorem 16.46, Theorem 16.50]{Janson}.

We denote the Wiener integral of a function $f\in L^2([0,\infty))$ by $I(f)$, which is the continuous linear extension of the mapping $\eins_{(0,t]}\mapsto B_t$ from $L^2([0,\infty))$ to $L^2(\Omega) := L^2(\Omega,\mathcal{F},P)$. Moreover, we denote the norm and inner product on $L^2([0,\infty))$ by $\|\cdot \|$ and $\langle \cdot, \cdot\rangle$ and by $\mathcal{E}$ the set of step functions on left half-open intervals, i.e., functions of the form
$$
g(x)= \sum\limits_{j=1}^{m} a_j \eins_{(b_j,c_j]}(x),\quad m\in \NN, a_j \in \RR, 0\leq b_j < c_j.
$$
For every $X \in L^2(\Omega)$ and $f \in L^2([0,\infty))$, the \emph{S-transform} of $X$ at $f$ is defined as
\[
(SX)(f) := \ex[X \exp\left(I(f) - \|f\|^2/2\right)]. 
\]
As the set of stochastic exponentials, also known as Wick exponentials,
\begin{equation}\label{eq:WickExponential}
\exp^{\diamond}(I(g)) := \exp\left(I(g) - \|g\|^2/2\right)\ , \quad g \in \mathcal{E} 
\end{equation}
is a total set in in $L^2(\Omega)$ (see e.g.\cite[Corollary 3.40]{Janson}), every random variable in $L^2(\Omega)$ is uniquely determined by its S-transform on $\mathcal{E}$, i.e. for $X, Y\in L^2(\Omega)$,
$
\forall g\in \mathcal{E} : (SX)(g)=(SY)(g)
$
implies $X=Y$ $P$-almost surely. The S-transform is a continuous and injective function on $L^2(\Omega)$ (see e.g. \cite[Chapter 16]{Janson} for more details). As an example, for $f,g \in L^2([0,\infty))$, we have $(S \ \exp\left(I(f) - \|f\|^2/2\right))(g) = \exp\left(\langle f,g\rangle\right)$. We denote for every $p\geq 1$, 
\[
L^p(\Omega\times [0,\infty)):= L^p(\Omega \times [0,\infty), \mathcal{F} \otimes \mathcal{B}([0,\infty)), P \otimes \lambda_{[0,\infty)}). 
\]
The characterization of random variables via the S-transform can be used to introduce the Hitsuda-Skorokhod integral,  cp. \cite[Theorem 16.46, Theorem 16.50]{Janson}:
\begin{definition}\label{def:Skorokhod}
The process $u = (u_t)_{t \geq 0}\in L^2(\Omega\times [0,\infty))$ is said to belong to the domain $D(\delta)$ of the Hitsuda-Skorokhod integral, if there is an $X\in L^2(\Omega)$ such that for every $g\in \mathcal{E}$ we have
$$
(SX)(g)=\int_0^\infty (Su_t)(g)g(t)dt.
$$
In this case, $X$ is uniquely determined and $\delta(u):=X$ is called the {\it Hitsuda-Skorokhod integral of $u$}.
\end{definition}

Moreover, the characterization via S-transform enables us to introduce stochastic integrals with respect to processes beyond semimartingales, like fractional Brownian motion, see e.g. \cite{Bender}.

The S-transform is closely related to a product imitating uncorrelated random variables as $\ex[X \diamond Y]=\ex[X]\ex[Y]$, which is implicitly contained in the Hitsuda-Skorokhod integral and a fundamental tool in stochastic analysis. Due to the injectivity of the S-transform, the \emph{Wick product} can be introduced via
\[
\forall f \in L^2([0,\infty)) \ : \ S(X \diamond Y)(f) = (S X)(f) (S Y)(f)
\]
on a dense subset in $L^{2}(\Omega) \times L^{2}(\Omega)$. For more details on Wick product we refer to \cite{Janson, Holden_Buch, Kuo}. In particular, for a Wiener Integral $I(f)$, Hermite polynomials play the role of monomials in standard calculus as $I(f)^{\diamond k} = h^k_{\|f\|^2}(I(f))$ and the Wick exponential can be reformulated as
\begin{equation}\label{eq:WickExWickPowerSeries}
\exp^{\diamond}(I(f)) =\sum_{k=0}^{\infty} \frac{1}{k!}I(f)^{\diamond k}. 
\end{equation}

Now we prove that the elementary Ayew-Kuo integral in Definition \ref{def:ElemNewIntegral} extends the Hitsuda-Skorokhod integral in Definition \ref{def:Skorokhod} in the setting of Definition \ref{def:ElemNewIntegral}.

\begin{theorem}\label{thm:ElementaryNewIntIsSkorohod}
Suppose $[a,b] \subset [0,\infty)$, an adapted $L^2$-continuous stochastic process $f$ and an instantly independent $L^2$-continuous process $\varphi$ such that the sequence
\begin{equation*}
\sum_{i=1}^{n} f(t_{i-1}) \varphi(t_i) \Delta B_{i},
\end{equation*}
converges strongly in $L^2(\Omega)$ as $\|\Delta_n\|$ tends to zero (which are stricter assumptions than in Definition \ref{def:ElemNewIntegral}). Then the limit $\bar{I}(f \varphi)$ equals the Hitsuda-Skorokhod integral $\delta(f \varphi)$.
\end{theorem}

The proof is reduced to the following characterization of strong convergence in terms of the S-transform. We will need
\begin{equation}\label{eq:MomentWickExp}
\ex[(e^{\diamond I(g)})^p] = \exp\left(\frac{p^2-p}{2}\|g\|^2\right) 
\end{equation}
for all $p >0$, see e.g.\cite[Cor. 3.38]{Janson}.

\begin{proposition}\label{prop:StransformL2}
Suppose $X, X_n \in L^2(\Omega)$ for every $n\in\NN$ and $\ex[(X_n)^2]  \rightarrow \ex[X^2]$. Then the following assertions are equivalent as $n$ tends to infinity:
 \begin{enumerate}
  \item[(i)] $X_n \rightarrow X$  strongly in $L^2(\Omega)$.
    \item[(ii)] 
     $(S X_n)(g) \rightarrow (S X)(g)$ for every $g\in \mathcal{E}$.
 \end{enumerate}
\end{proposition}

\begin{proof}
$(i) \rightarrow (ii)$: \ This is a direct consequence of the $L^2$-continuity of the S-transform by the Cauchy-Schwarz inequality and \eqref{eq:MomentWickExp},
\begin{equation}\label{eq:SIsCont}
|(S X_n)(g) - (S X)(g)| = |\ex[(X_n-X) \exp^{\diamond}(I(g))]| \leq \ex[(X_n-X)^2]^{1/2}\exp(\|g\|^2/2) \rightarrow 0. 
\end{equation}
$(ii) \rightarrow (i)$: \ Thanks to the total set of Wick exponentials \eqref{eq:WickExponential} and \cite[Theorem V. 1.3]{Yosida}, $X_n$ converges weakly to $X$ in the Hilbert space $L^2(\Omega)$. As the norms converge, we conclude strong convergence as well.
\end{proof}

\begin{proof}[Proof of Theorem \ref{thm:ElementaryNewIntIsSkorohod}]
 The Wiener-It\^o chaos expansion of the $\mathcal{F}_t$-measurable random variables $f(t)$ must be of the form
\begin{align*}
f(t) = \sum_{k\geq 0} I^k(f^k(t_1,\ldots, t_k,t)),
\end{align*}
where the integrands are symmetric functions in $L^2([a,b]^{k+1};\RR)$
$$
f^k(t_1,\ldots, t_k,t) = \eins_{\{t_1 \vee \cdots \vee t_k \leq t\}}f^k(t_1,\ldots, t_k,t),
$$ 
 (cf. \cite[Lemma 2.5.2]{Holden_Buch}). Similarly, due to the instant independence, the Wiener chaos expansion of $\varphi(t)$ is given by
$$
\varphi(t) = \sum_{k\geq 0} I^k(f_{\varphi}^k(t_1,\ldots, t_k,t)),
$$
where the integrands are
$$
f_{\varphi}^k(t_1,\ldots, t_k,t) = \eins_{\{t_1 \wedge \cdots \wedge t_k \geq t\}}f_{\varphi}^k(t_1,\ldots, t_k,t)
$$
(see e.g. the projection as an influence on the integrands in \cite[Lemma 1.2.5]{Nualart}). Hence, all chaoses of $f(t)$ and $\varphi(t)$ are based on disjoint increments of the underlying Brownian motion and we conclude by \cite[Proposition 2.4.2]{Holden_Buch} for all partitions
$$
f(t_{i-1}) \varphi(t_i) \Delta B_i = f(t_{i-1}) \diamond \varphi(t_i) \diamond \Delta B_i
$$
and for all $t \in [a,b]$,
\begin{equation}\label{AKIntegrandWick}
f(t) \varphi(t) = f(t) \diamond \varphi(t). 
\end{equation}

Therefore, for every partition and $g \in \mathcal{E}$, we have
\begin{align*}
\left(S \ \left(\sum_{i=1}^{n} f(t_{i-1}) \varphi(t_i) \Delta B_{i}\right)\right)(g) &= \sum_{i=1}^{n} \left(S ( f(t_{i-1}) \diamond \varphi(t_i) \diamond \Delta B_i)\right)(g)\\
&= \sum_{i=1}^{n} (S f(t_{i-1}))(g) (S \varphi(t_i))(g) (S \Delta B_i)(g) \\
&= \sum_{i=1}^{n} (S f(t_{i-1}))(g) (S \varphi(t_i))(g) (g(t_i) - g(t_{i-1})).
\end{align*}
The $L^2$-continuity of the stochastic process $f$ gives via \eqref{eq:SIsCont}
\begin{equation*}
|(S (f(x))(g) - (S f(x_n))(g)| \leq \ex[|f(x)-f(x_n)|^2]^{1/2}\exp(\|g\|^2/2) \rightarrow 0 
\end{equation*}
as $x_n \rightarrow x$. Analogously we conclude the continuity of $(S \varphi(\cdot))(g)$ for every fixed $g$. Therefore, the function $(S f(\cdot))(g) (S \varphi(\cdot))(g) g(\cdot)$ is piecewise continuous and Riemann-integrable and the Riemann sum $\sum_{i=1}^{n} (S f(t_{i-1}))(g) (S \varphi(t_i))(g) (g(t_i) - g(t_{i-1}))$ converges for $\| \Delta_n\| \rightarrow 0$ to the Riemann integral
$$
\int_{a}^{b} (S f(t))(g) (S \varphi(t))(g) g(t) dt.
$$
Thus, via \eqref{AKIntegrandWick} and Definition \ref{def:Skorokhod}, for all $g \in \mathcal{E}$ we conclude,
\begin{align*}
\lim_{n \rightarrow \infty} \left(S \ \sum_{i=1}^{n} f(t_{i-1}) \varphi(t_i) \Delta B_{i}\right)(g) &= \int_{a}^{b} (S f(t))(g) (S \varphi(t))(g) g(t) dt\\
&=\int_{a}^{b} (S f(t) \varphi(t))(g) g(t) dt = (S \, \delta(f \varphi))(g). 
\end{align*}
Thanks to strong convergence in \eqref{eq:NewIntegral1} and Proposition \ref{prop:StransformL2}, we conclude $\bar{I}(f\varphi) = \delta(f \varphi)$.
\end{proof}

\begin{remark}\label{rem:ElemNewIntNotSkor}
A simple example of an elementary Ayew-Kuo integral which is not a Hitsuda-Skorokhod integral, is given by the integrand $f(t) = \exp(B_t^2)$, $\varphi(t) = \exp((B_1-B_t)^2)$, i.e. the Ayew-Kuo integral
\[
\bar{I}\left(\exp(2B_t^2-2B_tB_1 +B_1^2))_{t \in [0,1]}\right)
\]
exists. Thanks to Fubini's theorem applied on a sum of nonnegative terms, $\frac{(2k-1)!!}{k!} = \frac{(2k)!}{2^k (k!)^2} = 2^k\binom{-1/2}{k}$ and Newton's binomial theorem, we have
\begin{align*}
\ex[\exp(p B_t^2)] = \sum_{k=0}^{\infty} \frac{p^k}{k!} \ex[B_t^{2k}] = \sum_{k=0}^{\infty} \frac{(2k-1)!!}{k!} (pt)^k = \sum_{k=0}^{\infty} \binom{-1/2}{k} (2pt)^k = \frac{1}{\sqrt{1-2pt}}. 
\end{align*}
Hence, $\exp(B_t^2)\in L^p(\Omega)$ if and only if \ $t \in [0,1/(2p)]$. Therefore, the integrand $\left(f(t) \varphi(t) \eins_{[0,1]}(t)\right)_{t \geq 0}$ is not in $D(\delta) \subset L^2(\Omega \times [0,\infty))$.
\end{remark}

\section{Extensions of the stochastic integrals}\label{section:Extensions}

The example in Remark \ref{rem:ElemNewIntNotSkor} indicates that the Ayew-Kuo integral extends the Hitsuda-Skorokhod integral due to the missing square-integrability. 

However, this disadvantage can be removed by the following extension of the Hitsuda-Skorokhod integral.

For a measure space $(\Omega,\mathcal{A},\mu)$, we denote by $L^0(\Omega,\mathcal{A},P)$ the space $(\Omega,\mathcal{A},\mu)$ with the topology of convergence in measure. 
Due to \cite[Corollary 3.40]{Janson}, for every $p \geq 0$, we have that the set of Wick exponentials \eqref{eq:WickExponential} is total in $L^p(\Omega,\mathcal{F},P)$. 

Moreover, as the set $\mathcal{E}$ is dense in $L^p([0,\infty), \mathcal{B}([0,\infty)),\lambda_{[0,\infty)})$ for every $p >0$, we conclude via a straightforward argument on products of dense sets the crucial fact:

\begin{proposition}\label{prop:Ldense}
For every $p \geq 0$, the linear span
\[
\mathcal{E}^{Exp} := {\rm lin}\left\{\exp^{\diamond}(I(g)) \otimes h\, : \; g, h\in \mathcal{E}\right\}. 
\]
is dense in the space $L^p(\Omega\times [0,\infty))$ (as a topological vector space).
\end{proposition}

Due to the totality of the Wick exponentials \eqref{eq:WickExponential} in $L^p(\Omega,\mathcal{F},P)$, the S-transform characterization of the Hitsuda-Skorokhod integral in Definition \ref{def:Skorokhod} can be extended to $L^p(\Omega\times [0,\infty))$ for $p >1$, see e.g. \cite[Section 16.4 and in particular Theorem 16.64]{Janson}:

\begin{definition}\label{def:SkorokhodLp}
The process $u = (u_t)_{t \geq 0}\in L^p(\Omega\times [0,\infty))$ is said to belong to the domain $D(\delta)$ of the Hitsuda-Skorokhod integral, if there is an $X\in L^p(\Omega,\mathcal{F},P)$ such that for every $g\in \mathcal{E}$
$$
(SX)(g)=\int_0^\infty (Su_t)(g)g(t)dt.
$$
In this case, $X$ is uniquely determined and $\delta(u):=X$ is called the {\it $L^p$-Hitsuda-Skorokhod integral of $u$}.
\end{definition}

Here, the condition $p >1$ ensures that all integrals involved exist, i.e. are finite and characterize the random elements uniquely.

Moreover, this approach can be extended to $L^1(\Omega\times [0,\infty))$ via an further transform (cf. \cite[Definition 16.52]{Janson}).

We observe the straightforward extension of Proposition \ref{prop:StransformL2} to $p \in (1,\infty)$:

\begin{proposition}\label{prop:StransformLp}
Suppose $X, X_n \in L^p(\Omega)$ for every $n\in\NN$ and $\ex[(X_n)^p]  \rightarrow \ex[X^p]$. Then the following assertions are equivalent as $n$ tends to infinity:
 \begin{enumerate}
  \item[(i)] $X_n \rightarrow X$  strongly in $L^p(\Omega)$.
    \item[(ii)] 
     $(S X_n)(g) \rightarrow (S X)(g)$ for every $g\in \mathcal{E}$.
 \end{enumerate}
\end{proposition}

\begin{proof}
$(i) \rightarrow (ii)$: It is again a direct consequence of the $L^p$-continuity of the S-transform via H\"older-inequality and \eqref{eq:MomentWickExp},
\begin{equation}\label{eq:SIsLpCont}
|(S X_n)(g) - (S X)(g)| \leq \ex[(X_n-X)^p]^{1/p}\exp\left(\frac{\|g\|^2}{2(p-1)}\right) \rightarrow 0. 
\end{equation}
$(ii) \rightarrow (i)$: \ $X_n$ converges weakly to $X$ in the normed space $L^p(\Omega)$ via \cite[Theorem V. 1.3]{Yosida}. As the norms converge, we conclude via \cite[Corollary 4.7.16]{Bogachev} norm convergence in $L^p(\Omega)$ as well.
\end{proof}

The we conclude similarly for every $p>1$:

\begin{theorem}\label{thm:ElementaryNewIntIsSkorohodLp}
Suppose $[a,b] \subset [0,\infty)$, an adapted continuous stochastic process $f$ and an instantly independent continuous process $\varphi$ such that the sequence
\begin{equation*}
\sum_{i=1}^{n} f(t_{i-1}) \varphi(t_i) \Delta B_{i},
\end{equation*}
converges strongly in $L^p(\Omega)$ as $\|\Delta_n\|$ tends to zero. Then the limit $\bar{I}(f \varphi)$ equals the Hitsuda-Skorokhod integral $\delta(f \varphi)$ in Definition \ref{def:SkorokhodLp}.
\end{theorem}

The proof follows the lines of the proof of Theorem \ref{thm:ElementaryNewIntIsSkorohod} making use of Proposition \ref{prop:StransformLp} and is omitted.\\

In a series of articles including \cite{HKSZ}, the elementary Ayew-Kuo integral is extended as follows:

\begin{definition}\label{def:NewIntegral}
Let $[a,b] \subset [0,\infty)$. Suppose a sequence $\Phi_n$ in the linear span of the integrands of the elementary Ayew-Kuo integral and a stochastic process $\Phi: \Omega\times [a,b] \rightarrow \RR$ satisfying the condition, as $n$ tends to infinity,
\[
\int_{a}^{b}|\Phi_n(t) - \Phi(t)|^2 dt \rightarrow 0 \quad a.s. 
\]
Then, provided the limit on the right hand side exists in probability, the \emph{Ayew-Kuo integral} of $\Phi$ is defined as 
\[
\bar{I}(\Phi) = \lim_{n \rightarrow \infty} \bar{I}(\Phi_n) .
\]
\end{definition}

As the characterization via S-transform is essentially based on the elements in $\mathcal{E}^{Exp}$, here we define for every $p \geq 1$:

\begin{definition}\label{def:NewSkorohod}
We define the Hitsuda-Skorokhod integral on the set $\mathcal{E}^{Exp}$
via 
\[
\delta(\exp^{\diamond}(I(g)) \otimes h) := \exp^{\diamond}\left(I(g)\right)\left(I(h) - \langle g,h\rangle\right) 
\]
and linearity and extend it by $L^p$-closedness to the domain $D(\delta) \subset L^p(\Omega\times [0,\infty))$.

By the $L^p$-closedness we obviously mean for a sequence of processes $(u^n)_{n \in \NN} \subset L^p(\Omega\times [0,\infty))$ such that $u^n \rightarrow u$ strongly in $L^p(\Omega\times [0,\infty))$ and the existing Hitsuda-Skorokhod integrals $\delta(u^n)$ converge strongly towards $X$ in $L^p(\Omega,\mathcal{F},P)$, we define 
\[
\delta (u) := X. 
\]
\end{definition}

\begin{remark}\label{rem:NewSkorohod}
$\left.\right.$

(i) \ The existence and uniqueness of the Hitsuda-Skorokhod integral in Definition \ref{def:NewSkorohod} on the domain $D(\delta)$ is a simple consequence of the construction via the set $\mathcal{E}^{Exp}$ and the ($P$-almost sure)-uniqueness of the $L^p$-limits $\delta (u) := X$.

(ii) \ Suppose $g,h,v \in \mathcal{E}$. Via $\exp^{\diamond}(I(g)) \exp^{\diamond}(I(v)) = e^{\langle g,v\rangle}\exp^{\diamond}(I(g+v))$ and the Wiener chaos expansion in \eqref{eq:WickExWickPowerSeries}, it is
\begin{align*}
\ex[e^{\diamond I(g)} I(h) e^{\diamond I(v)}] = e^{\langle g,v\rangle} \ex[e^{\diamond I(g+v)} I(h)] = e^{\langle g,v\rangle} \langle g+v,h\rangle
\end{align*}
and therefore
\begin{align}\label{eqExpAsWick}
\left(S \, e^{\diamond I(g)} \left(I(h) - \langle g,h\rangle\right) \right) (v) &= e^{\langle g,v\rangle} \left(\langle g+v,h\rangle - \langle g,h\rangle\right) = \int_{0}^{\infty} (S e^{\diamond I(g)})(v) h(s) v(s) ds.
\end{align}
Thus, Definition \ref{def:NewSkorohod} equals Definition \ref{def:SkorokhodLp} on $\mathcal{E}^{Exp}$. Due to $L^p$-continuity of the S-transform \eqref{eq:SIsLpCont} and Proposition \ref{prop:Ldense}, the equality follows on the domain $D(\delta)$ for $p>1$. Moreover, the product type formula \eqref{eqExpAsWick} yields the representation
\begin{equation*}
\delta(e^{\diamond I(g)} \otimes h) = e^{\diamond I(g)} \diamond I(h).
\end{equation*}

(iii) \ Obviously, the S-transform characterization in Definition \ref{def:SkorokhodLp} cannot be applied on random elements in $L^p$ for $p<1$, where we are beyond Banach spaces. 

Let us consider a simple example of a Hitsuda-Skorokhod integral in Definition \ref{def:NewSkorohod} for $p=3/2$. Let the nonadapted constant process
\[
u = \exp(B_{1/3}^2) \eins_{[0,1/3]}(t). 
\]
An approximation via elements elements in $\mathcal{E}^{Exp}$ can be obtained via the expansion $\exp(B_t^2) = \sum_{k=0}^{\infty} B_t^{2k}/k!$ and the approximation of monomials. This follows by the linear expansion of ordinary monomials $B_t^{2k}$ in terms of Hermite polynomials $h^l_t(B_t)$, $l \in \{0,\ldots, 2k\}$ and 
\[
\dfrac{\partial^l}{\partial w^l} \left. e^{\diamond w B_t} \right|_{w=0} = h^l_t(B_t),
\]
$P$-almost surely and in $L^p(\Omega)$, $p\geq 1$ (see e.g. \cite{BP1}). We omit these technical reformulations and notice that it suffices to consider also the processes $v^n = \sum_{k=0}^{n} \frac{1}{k!} B_{1/3}^{2k}\eins_{[0,1/3]}(t)$ instead of the approximating sequence in $\mathcal{E}^{Exp}$. These finite chaos elements exist in $L^1(\Omega \times [0,\infty))$. Thanks to $\delta(B_{1/3}^{2k} \eins_{[0,1/3]}(t)) = B_{1/3}^{2k+1} - \frac{2k}{3} B_{1/3}^{2k-1}$ via S-transform (see e.g. the Skorokhod integration by parts formula in \cite[Theorem 6.15]{DiNunno}), we obtain the short representation
\begin{align*}
\delta\left(v^n\right) =  \sum_{k=0}^{n} \frac{1}{k!} B_{1/3}^{2k+1} -  \frac{2}{3}\sum_{k=1}^{n} \frac{k}{k!}B_{1/3}^{2k-1} = \frac{1}{3}\sum_{k=0}^{n} \frac{1}{k!} B_{1/3}^{2k+1}.
\end{align*}
A simple calculution gives for all $t\geq 0$, $k \in \NN$,
\begin{equation*}
\ex[|B_t^{2k+1}|] = \sqrt{2/\pi} (2k)!! t^{k+1/2}.  
\end{equation*}
Hence, by the triangle inequality
\[
\ex[|\delta\left(v^n\right)|]  \leq \frac{\sqrt{2/\pi}}{3} \sum_{k=0}^{n} \frac{(2k)!!}{k!}  \frac{1}{3^{k+1/2}} = \frac{\sqrt{2/\pi}}{3^{3/2}} \sum_{k=0}^{n}  (2/3)^{k}.
\]
Thanks to dominated convergence we conclude $u \in D(\delta)$ with $p=1$ (according to Definition \ref{def:NewSkorohod}) and
\[
\delta(u) =  \frac{1}{3}\sum_{k=0}^{\infty} \frac{1}{k!} B_{1/3}^{2k+1} = \frac{1}{3} B_{1/3}\exp(B_{1/3}^2).
\]
As observed in Remark \ref{rem:ElemNewIntNotSkor}, $u \notin L^p(\Omega\times [0,\infty))$ for all $p>3/2$ and therefore $u \notin D(\delta)$ for $p>3/2$.

%A simple calculution gives for all $t,\alpha > 0$, \begin{align}\ex[|B_t|^{\alpha}] &= 2\int_{0}^{\infty} \frac{1}{\sqrt{2\pi t}} x^{\alpha} e^{-x^2/(2t)} dx = 2\int_{0}^{\infty} \frac{1}{\sqrt{2\pi t}} (\sqrt{2tu})^{\alpha} e^{-u} \sqrt{t/(2u)}du\\&= \sqrt{\frac{(2t)^{\alpha}}{\pi}} \int_{0}^{\infty} u^{(\alpha-1)/2} e^{-u} du\\&=\sqrt{\frac{(2t)^{\alpha}}{\pi}} \Gamma((\alpha + 1)/2).\end{align} Hence, we conclude\begin{align*}\ex[\exp(p |B_t|^{\alpha})] = \sum_{k=0}^{\infty} \frac{1}{k!} p^k \ex[|B_t|^{\alpha k}] = \frac{1}{\sqrt{\pi}}\sum_{k=0}^{\infty} (p (2t)^{\alpha/2})^{k}\frac{\Gamma(\alpha k +1/2)}{k!}.\end{align*}

%(iv) \ %Example of Skorohod on L^0 
\end{remark}

The dense set in Proposition \ref{prop:Ldense} and the S-transform can be used to define the Malliavin derivative, see e.g. \cite[Proposition 49]{BP3}. Similarly to Definition \ref{def:NewSkorohod}, the following gives the Malliavin derivative on the Sobolev spaces $\mathbb{D}^{1,p}$, $p \geq 1$, in \cite[1.2]{Nualart}:
\begin{definition}\label{def:NewMalliavin}
We define the Malliavin derivative on the (dense) set $\{e^{\diamond I(g)}, g \in \mathcal{E}\}$
via 
\[
D_t (e^{\diamond I(g)}) := e^{\diamond I(g)} g(t) 
\]
and extend it to the closable operator on the domain $\mathbb{D}^{1,p}$.
\end{definition}

Following the extension in Definition \ref{def:NewSkorohod}, we introduce for every $p\geq 1$:

\begin{definition}\label{def:VeryNewIntegral}
We define the Ayew-Kuo integral on the set $\mathcal{E}^{Exp}$
via 
\[
\bar{I}(\exp^{\diamond}(I(g)) \otimes h) := \exp^{\diamond}\left(I(g)\right)\left(I(h) - \langle g,h\rangle\right) 
\]
and linearity and extend it by the following closedness to the domain $D(\bar{I}) \subset L^p(\Omega\times [0,\infty))$:

Suppose a sequence of processes $(u^n)_{n \in \NN} \subset \mathcal{E}^{Exp}$ such that 
\[
\int_{0}^{\infty}|u^n_s - u_s|^p ds \rightarrow 0 \quad a.s. , 
\]
and the existing Ayew-Kuo integrals $\bar{I}(u^n)$ converge in probability, then we define
\[
\bar{I}(u) = \lim_{n \rightarrow \infty} \bar{I}(u^n). 
\]
\end{definition}

\begin{remark}
(i) \ The existence and uniqueness of the stochastic integral in Definition \ref{def:VeryNewIntegral} on the domain $D(\bar{I})$ follows again by the construction on the dense set $\mathcal{E}^{Exp}$ and the $P$-a.s.-uniqueness of the limits.

(ii) \ In contrast to Definitions \ref{def:ElemNewIntegral} and \ref{def:NewIntegral},  the stochastic integral in Definition \ref{def:VeryNewIntegral} is defined in one step on the time horizon $[0,\infty)$.
%Examples of integrals very new integral exists but old new integral does not exist
\end{remark}

Our main result is the following:

\begin{theorem}
$\left.\right.$
\begin{enumerate}
 \item Definition \ref{def:VeryNewIntegral} extends Definition \ref{def:NewSkorohod}.
\item Definition \ref{def:VeryNewIntegral} extends Definition \ref{def:NewIntegral}.  
\end{enumerate}
\end{theorem}

\begin{proof}
(i) \ Suppose $u$ in the domain of the Hitsuda-Skorokhod integral in Definition \ref{def:NewSkorohod}, i.e. there exists a $p \geq 1$ and a sequence $(u^n)_{n \in \NN} \subset \mathcal{E}^{Exp}$ with $u^n \rightarrow u$ in $L^p(\Omega\times [0,\infty))$ and $\delta(u^n) \rightarrow X$ in $L^p(\Omega)$. This implies the convergence $u^n \rightarrow u$ in measure $P \otimes \lambda$ and therefore a subsequence $(n_k)_{k \in \NN}$ such that $u^{n_k} \rightarrow u$ $P \otimes \lambda$-almost everywhere as $k$ tends to infinity. Due to $\delta(u^n) = \bar{I}(u^n)$ we conclude $\delta(u^{n_k}) \rightarrow \delta(u)$ in probability. Since the limit in probability is unique, this yields that $u$ is contained in the domain of the Ayew-Kuo integral in Definition \ref{def:VeryNewIntegral} and $\bar{I}(u) = \delta(u)$.

(ii) \ Firstly, for every element $g(x)= \sum\limits_{j=1}^{m} a_j \eins_{(b_j,c_j]}(x)$ in $\mathcal{E}$, we observe the reformulation 
\begin{align*}
\exp^{\diamond}(I(g)) &= \exp^{\diamond}(I(g\eins_{[0,t)})) \diamond \exp^{\diamond}(I(g\eins_{(t,\infty)}))= \exp^{\diamond}(I(g\eins_{[0,t)})) \exp^{\diamond}(I(g\eins_{(t,\infty)})).
\end{align*}
Here the Wick product equals the ordinary product since the Wick exponentials are based on disjoint increments of the underlying Brownian motion (cf. \cite[Lemma 2.5.2]{Holden_Buch}). Hence, for the integrand $e^{\diamond I(g)} h$ we conclude an adapted part 
$$
\exp^{\diamond}(I(g\eins_{[0,t)}) h(t) = \exp\left(\sum_{j=1}^{m} a_j (B_{c_j \wedge t} - B_{b_j \wedge t}) - \frac{1}{2}\sum_{j=1}^{m} a_j^2 ((c_j \wedge t) - (b_j \wedge t)\right) h(t) 
$$
and an instantly independent part 
$$
\exp^{\diamond}(I(g\eins_{(t,\infty)})) = \exp\left(\sum_{j=1}^{m} a_j (B_{c_j \vee t} - B_{b_j \vee t}) - \frac{1}{2}\sum_{j=1}^{m} a_j^2 ((c_j \vee t) - (b_j \vee t)\right).
$$
Due to a standard density argument in $L^2([0,\infty))$, we obtain that for all $g,h \in \mathcal{E}$ the process $\exp^{\diamond}(I(g)) h(\cdot)$, is contained in the domain of the Ayew-Kuo integral in Definition \ref{def:NewIntegral}. In particular, thanks to linearity of the integral on disjoint time intervals, we conclude by a simple computation on finite sums in \eqref{eq:NewIntegral1},
$$
\bar{I}(\exp^{\diamond}(I(g)) \otimes h) = \exp\left(I(g)-\frac{1}{2}\int_{0}^{1} g^2(s) ds\right)\left(I(h) - \int_{0}^{1} g(s) h(s) ds\right)
%= e^{\diamond I(g)} (I(h) - \langle g,h\rangle)
.
$$
Hence, $\mathcal{E}^{Exp}$ is contained in the domain of the stochastic integral in Definition \ref{def:NewIntegral}. Then, via the (a.s)-$P$-closedness we conclude that Definition \ref{def:VeryNewIntegral} extends Definition \ref{def:NewIntegral}.
\end{proof}

\begin{remark}
Comparing Definitions \ref{def:NewSkorohod} and \ref{def:VeryNewIntegral}, we observe the difference of the closedness: In the Hitsuda-Skorokhod the first assumed convergence of the integrands in $L^p$-sense can be stronger than the almost sure convergence in Definition \ref{def:VeryNewIntegral}, whereas the postulated convergence  in probability of the integrals in Definition \ref{def:VeryNewIntegral} is weaker than for the Hitsuda-Skorokhod integral. However, we are unable to construct an example of an integrand in $D(\bar{I}) \setminus D(\delta)$ for the same $p \geq 1$. Therefore, we assume that these stochastic integrals coincide in the usual world of stochastic processes.
\end{remark}

\begin{remark}
The extension of the stochastic integrals in Definition \ref{def:VeryNewIntegral} or \ref{def:NewSkorohod} to $0<p<1$ seems difficult for the following reasons. Firstly, the spaces $L^p(\Omega\times [0,\infty))$ are indeed topological vector spaces but not Banach spaces anymore. In particular their dual is the trivial space $\{0\}$ (cf. \cite{Day}) and therefore a stochstic integral constructed by the set $\mathcal{E}^{Exp}$ and linearity as in Definition \ref{def:VeryNewIntegral} must be an unbounded operator. This would be extremely impractical. 
\end{remark}


\begin{thebibliography}{plain}


\bibitem{AyedKuo1} Ayed, W. and Kuo, H.-H. An extension of the It\^o integral. \textit{Commun. Stoch. Anal.} \textbf{2} (3), (2008) 323--333.

\bibitem{AyedKuo2} Ayed, W. and Kuo, H.-H. An extension of the It\^o integral: toward a general theory of stochastic integration. \textit{Theory Stoch. Process.} \textbf{16} (1), (2010) 17--38.

%\bibitem{Aigner} Aigner, M. \textit{A course in enumeration} Graduate Texts in Mathematics, 238. Springer, Berlin, 2007.

%\bibitem{Avram_Taqqu} F. Avram and M. Taqqu, Noncentral limit theorems and Appell polynomials. \textit{Ann. Probab.} \textbf{15} (2) (1987) 767--775. 

\bibitem{Bender} Bender, C.
An S-transform approach to integration with respect to a fractional
Brownian motion.
\textit{Bernoulli} \textbf{6}, (2003) 955--983.


\bibitem{BP1} Bender, C. and Parczewski, P. Approximating a geometric fractional Brownian motion and related processes via discrete Wick calculus. \textit{Bernoulli} \textbf{16} (2), (2010) 389--417.

\bibitem{BP3} Bender, C. and Parczewski, P.
Discretizing Malliavin calculus. To appear in \textit{Stochastic Process. Appl.} (2017)

\bibitem{Bogachev} Bogachev, V.I.
\textit{Measure Theory.}
Springer, Berlin, 2007.

%\bibitem{Buckdahn_Nualart} Buckdahn, R. and Nualart, D. Linear stochastic differential equations and Wick products. \textit{Probab. Theory Related Fields} \textbf{99} (4) (1994) 501--526.

\bibitem{Day} Day, M. M. The spaces $L^p$ with $0<p<1$. \textit{Bull. Amer. Math. Soc.} \textbf{46},  (1940) 816--823.

\bibitem{DiNunno} Di Nunno, G. and {\O}ksendal, B. and Proske, F. \textit{Malliavin calculus for L\'{e}vy processes with applications to finance} Universitext. Springer, Berlin, 2009.

%\bibitem{Fournie} Fourni\'e, E. and Lasry, J.-M. and Lebuchoux, J. and Lions, P.-L. and Touzi, N. Applications of Malliavin calculus to Monte Carlo methods in finance. \textit{Finance Stoch.} \textbf{3} (4) (1999) 391--412.



\bibitem{Holden_Buch} Holden H. and {\O}ksendal, B. and Ub{\o}e, J. and Zhang, T.
\textit{Stochastic Partial Differential Equations. A Modeling, White Noise Functional Approach. Second Edition} Springer, New York, 2010.

\bibitem{HKSZ} Hwang, C.-R. and Kuo, H.-H. and Sait\^o, K. and Zhai, J. A general It\^o formula for adapted and instantly independent stochastic processes. \textit{Commun. Stoch. Anal.} \textbf{10} (3),  (2016) 341--362.

%\bibitem{Jentzen_MG_Y} Jentzen, A. and M\"uller-Gronbach, T. and Yaroslavtseva, L. On stochastic differential equations with arbitrary slow convergence rates for strong approximation. \textit{Commun. Math. Sci.} \textbf{14} (6) (2016) 1477--1500.


\bibitem{Janson} Janson, S.
\textit{Gaussian Hilbert Spaces.}, Cambridge University Press, Cambridge, 1997.


%\bibitem{Karatzas_Shreve} Karatzas, I. and Shreve, S. E. \textit{Brownian motion and stochastic calculus.}, Second edition. Graduate Texts in Mathematics, 113. Springer. New York, 1991.

%\bibitem{Kloeden_Platen} Kloeden, P. and Platen, E. \textit{Numerical solution of stochastic differential equations.} Applications of Mathematics, 23. Springer-Verlag, Berlin, 1992.


\bibitem{Kuo} Kuo, H.-H.
\textit{White Noise Distribution Theory.}
Probability and Stochastics Series. CRC Press, Boca Raton, 1996.

\bibitem{KuoStochInt} Kuo, H.-H.
\textit{Introduction to Stochastic Integration.}
Springer, New York, 2006.

\bibitem{KuoNewInt} Kuo, H.-H. The It\^o calculus and white noise theory: a brief survey toward general stochastic integration. \textit{Commun. Stoch. Anal.} \textbf{8}, (1), (2014)  111--139.

%\bibitem{Mueller_Gronbach} M\"uller-Gronbach, T. Optimal pointwise approximation of SDEs based on Brownian motion at discrete points. \textit{Ann. Appl. Probab.} \textbf{14}, (4) (2004) 1605--1642.


%\bibitem{NP} Neuenkirch, A. and Parczewski, P. Optimal approximation of Skorokhod integrals. To appear in \textit{J. Theoret. Probab.} (2016).


\bibitem{Nualart} Nualart, D.
\textit{The Malliavin Calculus and Related Topics.} Second Edition. Probability and its Applications. Springer, New York, 2006. 

\bibitem{NualartPardoux} Nualart, D. and Pardoux, \'E.
Stochastic calculus with anticipating integrands. \textit{Probab. Theory Related Fields} \textbf{78}(4), (1988) 535--581.

%\bibitem{P} Parczewski, P. A Wick functional limit theorem. \textit{Probab. Math. Statist.} \textbf{34} (1), 127--145, (2014).

%\bibitem{P2} Parczewski, P. Optimal approximation of Skorokhod integrals - examples with substandard rates. \textit{ Commun. Stoch. Anal.} \textbf{11} (1), (2017) 43--61.

 %\bibitem{Przybylowicz} Przyby\l owicz, P. Optimal sampling design for approximation of stochastic It\^o integrals with application to the nonlinear Lebesgue integration. \textit{J. Comput. Appl. Math.} \textbf{245} (2013) 10--29. 


%\bibitem{Weyl} Weyl, H. \"Uber die Gleichverteilung von Zahlen mod. Eins. \textit{Math. Ann.} \textbf{77}, (3) (1916) 313--352.

\bibitem{Yosida} Yosida, K. 
\textit{Functional analysis.} Reprint of the sixth (1980) edition. 
Classics in Mathematics. Springer, Berlin, 1995.
\end{thebibliography}
\end{document}